\documentclass[english]{amsart}
\usepackage[T1]{fontenc}
\usepackage[latin9]{inputenc}
\usepackage{amsthm}
\usepackage{amsmath}
\usepackage{amssymb}
\usepackage{esint}
\makeatletter
\theoremstyle{plain}
\newtheorem{thm}{\protect\theoremname}
  \theoremstyle{definition}
  \newtheorem{defn}[thm]{\protect\definitionname}
  \theoremstyle{plain}
  \newtheorem{lem}[thm]{\protect\lemmaname}
  \theoremstyle{remark}
  \newtheorem{rem}[thm]{\protect\remarkname}
  \theoremstyle{plain}
  \newtheorem{prop}[thm]{\protect\propositionname}

\DeclareMathOperator{\cat}{cat}

\makeatother

\usepackage{babel}
  \providecommand{\definitionname}{Definition}
  \providecommand{\lemmaname}{Lemma}
  \providecommand{\propositionname}{Proposition}
  \providecommand{\remarkname}{Remark}
\providecommand{\theoremname}{Theorem}

\begin{document}
\title{A multiplicity result for double singularly perturbed elliptic systems}

\thanks{The first author was  supported by Gruppo Nazionale per l'Analisi Matematica, la Probabilit\`{a} e le loro 
Applicazioni (GNAMPA) of Istituto Nazionale di Alta Matematica (INdAM)}

\author{Marco Ghimenti}
\address{M. Ghimenti, \newline Dipartimento di Matematica Universit\`a di Pisa
Largo B. Pontecorvo 5, 56100 Pisa, Italy}
\email{ghimenti@mail.dm.unipi.it}

\author{Anna Maria Micheletti}
\address{A. M. Micheletti, \newline Dipartimento di Matematica Universit\`a di Pisa
Largo B. Pontecorvo 5, 56100 Pisa, Italy}
\email{a.micheletti@dma.unipi.it.}

\begin{abstract}
We show that the number of low energy solutions of a double singularly
perturbed Schroedinger Maxwell system type on a smooth 3 dimensional
manifold $(M,g)$ depends on the topological properties of the manifold.
The result is obtained via Lusternik Schnirelmann category theory.
\end{abstract}



\maketitle

\section{Introduction}

Given real numbers $q>0$, $\omega>0$, $p>4$, we consider the following
system of Schroedinger Maxwell type on a smooth manifold $M$ endowed
with a Riemannian metric $g$
\begin{equation}
\left\{ \begin{array}{cc}
-\varepsilon^{2}\Delta_{g}u+u+\omega uv=|u|^{p-2}u & \text{ in }M\\
-\varepsilon^{2}\Delta_{g}v+v=qu^{2} & \text{ in }M\\
u>0\text{ in }M
\end{array}\right.\label{eq:sms}
\end{equation}
where $\Delta_{g}$ is the Laplace-Beltrami operator on $M$. 

We want to prove that when the parameter $\varepsilon$ is sufficiently
small, there are many low energy solution of (\ref{eq:sms}). In particular
the number of solutions of (\ref{eq:sms}) is related to the topology
of the manifold $M$. We suppose without loss of generality, that
the manifold $M$ is isometrically embedded in $\mathbb{R}^{n}$ for
some $n$.

Here there is a competition between the two equation, since both share
the same singular perturbation of order $\varepsilon^{2}$. In \cite{GM1,GMruf}
we dealt with a similar system where only  the first equation had
a singular perturbation. In this case the second equation disappears
in the limit. In Section \ref{sub:The-limit-problem} we write the
limit problem taking care of the competition, and we find the model
solution for system (\ref{eq:sms}). 

A problem similar to (\ref{eq:sms}), namely the Schroedinger-Newton
system, has been studied from a dynamical point of view in \cite{PieMar}.
Also in this paper the two equation have the $\varepsilon^{2}$ singular
perturbation. 

Recently, Schroedinger Maxwell type systems received considerable attention
from the mathematical community, we refer, e.g. to \cite{AR,AP,BF,Co,DM,Dav,K,RuJFA}. 
A special case of Schroedinger Maxwell type systems, namely when the system is set in $\mathbb{R}^{3}$, 
takes the name of Schroedinger-Poisson-Slater equation and it arises 
in  Slater approximation of the Hartree-Fock
model. 
We want here to especially mention some result of the existence of solutions, i.e. \cite{AP,DM,He,K,RuARMA}, 
since the limit problem (\ref{eq:sms}) is a Schroedinger-Poisson-Slater type equation.
(for a more exhaustive discussion on Schroedinger-Poisson-Slater and on the 
physical models that leads to this equation we  we refer to \cite{Ia,IR}
and the references therein). 

Our main results is the following.
\begin{thm}
\label{thm:1}Let $4<p<6$. For $\varepsilon$ small enough there
exist at least $\cat(M)$ positive solutions of (\ref{eq:sms}). 
\end{thm}
Here we recall the definition of the Lusternik Schnirelmann category
of a set.
\begin{defn}
Let $X$ a topological space and consider a closed subset $A\subset X$.
We say that $A$ has category $k$ relative to $X$ ($\cat_{X}A=k$)
if $A$ is covered by $k$ closed sets $A_{j}$, $j=1,\dots,k$, which
are contractible in $X$, and $k$ is the minimum integer with this
property. We simply denote $\cat X=\cat_{X}X$.
\end{defn}

\section{Preliminary results}

We endow $H^{1}(M)$ and $L^{p}(M)$ with the following equivalent
norm 
\begin{eqnarray*}
\|u\|_{\varepsilon}^{2}=\frac{1}{\varepsilon^{3}}\int_{M}\varepsilon^{2}|\nabla u|^{2}+u^{2}d\mu_{g} &  & |u|_{\varepsilon,p}^{p}=\frac{1}{\varepsilon^{3}}\int_{M}|u|^{p}d\mu_{g}\\
\|u\|_{H^{1}}^{2}=\int_{M}|\nabla u|^{2}+v^{2}d\mu_{g} &  & |u|_{p}^{p}=\int_{M}|u|^{p}d\mu_{g}
\end{eqnarray*}
and we refer to $H_{\varepsilon}$ (resp. $L_{\varepsilon}^{p}$)
as space $H^{1}(M)$ (resp. $L_{\varepsilon}^{p}$) endowed with the
$\|\cdot\|_{\varepsilon}$ (resp. $|\cdot|_{\varepsilon,p}$)norm.
Obviously, we refer to the scalar product on $H_{\varepsilon}$ as
\[
\left\langle u,v\right\rangle _{\varepsilon}=\frac{1}{\varepsilon^{3}}\int_{M}\varepsilon^{2}\nabla u\nabla v+uvd\mu_{g}.
\]
Following an idea by Benci and Fortunato \cite{BF}, for any $\varepsilon$
we introduce the map $\psi_{\varepsilon}:H^{1}(M)\rightarrow H^{1}(M)$
that is the solution of the equation
\begin{equation}
-\varepsilon^{2}\Delta_{g}v+v=qu^{2}\text{ in }M\label{eq:psi-eps}
\end{equation}

\begin{lem}
\label{lem:psi}The map $\psi:H^{1}(M)\rightarrow H^{1}(M)$ is of
class $C^{2}$ with derivatives $\psi'(u)$ and $\psi''(u)$ which
satisfy 
\begin{eqnarray}
-\varepsilon^{2}\Delta_{g}\psi_{\varepsilon}'(u)[\varphi]+\psi_{\varepsilon}'(u)[\varphi] & = & 2qu\varphi\label{eq:derprima}\\
-\varepsilon^{2}\Delta_{g}\psi_{\varepsilon}''(u)[\varphi_{1},\varphi_{2}]+\psi_{\varepsilon}''(u)[\varphi_{1},\varphi_{2}] & = & 2q\varphi_{1}\varphi_{2}\label{eq:derseconda}
\end{eqnarray}
for any $\varphi,\varphi_{1},\varphi_{2}\in H^{1}(M)$. Moreover $\psi_{\varepsilon}(u)\ge0$. \end{lem}
\begin{proof}
The proof is standard.\end{proof}
\begin{rem}
\label{rem:tiquadro}We observe that by simple computation, for any
$t>0$ we have that $\psi_{\varepsilon}(tu)=t^{2}\psi_{\varepsilon}(u)$.
In fact, if $\psi_{\varepsilon}(u)$ solves (\ref{eq:psi-eps}), multiplying
by $t^{2}$ both sides of (\ref{eq:psi-eps}) we get the claim.\end{rem}
\begin{lem}
\label{lem:Tder}The map $T_{\varepsilon}:H_{\varepsilon}\rightarrow\mathbb{R}$
given by
\[
T_{\varepsilon}(u)=\int_{M}u^{2}\psi_{\varepsilon}(u)d\mu_{g}
\]
is a $C^{2}$ map and its first derivative is 
\[
T_{\varepsilon}'(u)[\varphi]=4\int_{M}\varphi u\psi_{\varepsilon}(u)d\mu_{g}.
\]
\end{lem}
\begin{proof}
The regularity is standard. The first derivative is 
\[
T_{\varepsilon}'(u)[\varphi]=2\int u\varphi\psi_{\varepsilon}(u)+\int u^{2}\psi_{\varepsilon}'(u)[\varphi].
\]
By (\ref{eq:derprima}) and (\ref{eq:psi-eps}) we have 
\begin{align*}
2\int u\varphi\psi_{\varepsilon}(u) & =\frac{1}{q}\left(-\varepsilon^{2}\int\Delta(\psi_{\varepsilon}'(u)[\varphi])\psi_{\varepsilon}(u)+\int\psi_{\varepsilon}'(u)[\varphi]\psi_{\varepsilon}(u)\right)\\
 & =\frac{1}{q}\left(-\varepsilon^{2}\int\psi_{\varepsilon}'(u)[\varphi]\Delta\psi_{\varepsilon}(u)+\int\psi_{\varepsilon}'(u)[\varphi]\psi_{\varepsilon}(u)\right)\\
 & =\int\psi_{\varepsilon}'(u)[\varphi]u^{2}
\end{align*}
and the claim follows.
\end{proof}
At this point we consider the following functional $I_{\varepsilon}\in C^{2}(H_{\varepsilon},\mathbb{R})$.
\begin{equation}
I_{\varepsilon}(u)=\frac{1}{2}\|u\|_{\varepsilon}^{2}+\frac{\omega}{4}G_{\varepsilon}(u)-\frac{1}{p}|u^{+}|_{\varepsilon,p}^{p}\label{eq:ieps}
\end{equation}
where 
\[
G_{\varepsilon}(u)=\frac{1}{\varepsilon^{3}}\int_{\Omega}u^{2}\psi_{\varepsilon}(u)dx=\frac{1}{\varepsilon^{3}}T_{\varepsilon}(u).
\]
By Lemma \ref{lem:Tder} we have 
\[
I_{\varepsilon}'(u)[\varphi]=\frac{1}{\varepsilon^{3}}\int_{\Omega}\varepsilon^{2}\nabla u\nabla\varphi+u\varphi+\omega u\psi_{\varepsilon}(u)\varphi-(u^{+})^{p-1}\varphi
\]
\[
I_{\varepsilon}'(u)[u]=\|u\|_{\varepsilon}^{2}+\omega G_{\varepsilon}(u)-|u^{+}|_{\varepsilon,p}^{p}
\]
 then if $u$ is a critical points of the functional $I_{\varepsilon}$
the pair of positive functions $(u,\psi_{\varepsilon}(u))$ is a solution
of (\ref{eq:sms}).

We define the following Nehari set
\[
{\mathcal N}_{\varepsilon}=\left\{ u\in H^{1}(M)\smallsetminus0\ :\ N_{\varepsilon}(u):=I'_{\varepsilon}(u)[u]=0\right\} 
\]
The Nehari set has the following properties (for a complete proof
see \cite{GM1})
\begin{lem}
\label{lem:nehari}If $p>4$, ${\mathcal N}_{\varepsilon}$ is a $C^{2}$
manifold and $\inf_{{\mathcal N}_{\varepsilon}}\|u\|_{\varepsilon}>0$. 

If $u\in{\mathcal N}_{\varepsilon}$, then 
\begin{align}
I_{\varepsilon}(u) & =\left(\frac{1}{2}-\frac{1}{p}\right)\|u\|_{\varepsilon}^{2}+\omega\left(\frac{1}{4}-\frac{1}{p}\right)G_{\varepsilon}(u)\nonumber \\
 & =\left(\frac{1}{2}-\frac{1}{p}\right)|u^{+}|_{p,\varepsilon}^{p}-\frac{\omega}{4}G_{\varepsilon}(u)\label{eq:I-nehari}\\
 & =\frac{1}{4}\|u\|_{\varepsilon}^{2}+\left(\frac{1}{4}-\frac{1}{p}\right)|u^{+}|_{p,\varepsilon}^{p}.\nonumber 
\end{align}
and it holds Palais-Smale condition for the functional $I_{\varepsilon}$
on ${\mathcal N}_{\varepsilon}$. 

Finally, for all $w\in H^{1}(M)$ such that $|w^{+}|_{\varepsilon,p}=1$
there exists a unique positive number $t_{\varepsilon}=t_{\varepsilon}(w)$
such that $t_{\varepsilon}(w)w\in{\mathcal N}_{\varepsilon}$. The number
$t_{\varepsilon}$ is the critical point of the function 
\[
H(t)=I_{\varepsilon}(tw)=\frac{1}{2}t^{2}\|w\|_{\varepsilon}^{2}+\frac{t^{4}}{4}\omega G_{\varepsilon}(w)-\frac{t^{p}}{p}.
\]

\end{lem}

\subsection{\label{sub:The-limit-problem}The limit problem }

Consider the following problem in the whole space.
\begin{equation}
\left\{ \begin{array}{cc}
-\Delta u+u+\omega uv=|u|^{p-2}u & \text{ in }\mathbb{R}^{3}\\
-\Delta v+v=qu^{2} & \text{ in }\mathbb{R}^{3}\\
u>0\text{ in }\mathbb{R}^{3}
\end{array}\right.\label{eq:PL}
\end{equation}
In an analogous way we define the function $\psi_{\infty}(u)$ as
a solution of the second equation and, as before, we can define a
functional
\[
I_{\infty}(u)=\frac{1}{2}\|u\|_{H^{1}}^{2}+\frac{\omega}{4}G(u)-\frac{1}{p}|u^{+}|_{p}^{p}
\]
where $G(u)=\int_{\mathbb{R}^{3}}u^{2}\psi_{\infty}(u)dx$ and the
Nehari manifold $\mathcal{N}_{\infty}=\left\{ u\in H^{1}(\mathbb{R}^{3})\smallsetminus0\ :\ I_{\infty}'(u)[u]=0\right\} .$
It is easy to prove (see \cite{He}) that the value 
\[
m_{\infty}=\inf_{\mathcal{N}_{\infty}}I_{\infty}
\]
is attained at least by a function $U$ which is a solution of problem
(\ref{eq:PL}). 

We will refer at problem (\ref{eq:PL}) as the limit problem. We set
\[
U_{\varepsilon}(x)=U\left(\frac{x}{\varepsilon}\right)
\]
and the function $U_{\varepsilon}$ will be the model solution for
a solution of problem (\ref{eq:sms}).

\section{Main ingredient of the proof}

We sketch the proof of Theorem \ref{thm:1}. First of all, it is easy
to see that the functional $I_{\varepsilon}\in C^{2}$ is bounded
below and satisfies Palais Smale condition on the complete $C^{2}$ manifold
${\mathcal N}_{\varepsilon}$. Then we have, by well known results, that
$I_{\varepsilon}$ has at least $\cat I_{\varepsilon}^{d}$ critical
points in the sublevel
\[
I_{\varepsilon}^{d}=\left\{ u\in H^{1}\ :\ I_{\varepsilon}(u)\le d\right\} .
\]
We prove that, for $\varepsilon$ and $\delta$ small enough, it holds
\[
\cat M\le\cat\left({\mathcal N}_{\varepsilon}\cap I_{\varepsilon}^{m_{\infty}+\delta}\right)
\]
where $m_{\infty}$ has been defined in the previous section. 

To get the inequality $\cat M\le\cat\left({\mathcal N}_{\varepsilon}\cap I_{\varepsilon}^{m_{\infty}+\delta}\right)$
we build two continuous operators
\begin{eqnarray*}
\Phi_{\varepsilon} & : & M\rightarrow{\mathcal N}_{\varepsilon}\cap I_{\varepsilon}^{m_{\infty}+\delta}\\
\beta & : & {\mathcal N}_{\varepsilon}\cap I_{\varepsilon}^{m_{\infty}+\delta}\rightarrow M^{+}.
\end{eqnarray*}
where 
\[
M^{+}=\left\{ x\in\mathbb{R}^{n}\ :\ d(x,M)<R\right\} 
\]
with $R$ small enough so that $\cat(M^{+})=\cat(M)$. Without loss
of generality, we can suppose $R=r$ the injectivity radius of $M$,
in order to simplify the notations.

Following an idea in \cite{BC1}, we build these operators $\Phi_{\varepsilon}$
and $\beta$ such that $\beta\circ\Phi_{\varepsilon}:M\rightarrow M^{+}$
is homotopic to the immersion $i:M\rightarrow M^{+}$. By a classical
result on topology (which we summarize in Remark \ref{rem:cat}) we
have 
\[
\cat M\le\cat\left({\mathcal N}_{\varepsilon}\cap I_{\varepsilon}^{m_{\infty}+\delta}\right)
\]
and the first claim of Theorem \ref{thm:1} is proved.
\begin{rem}
\label{rem:cat}Let $X_{1}$ and $X_{2}$, $X_{3}$ be topological
spaces with $X_{1}$ and $X_{3}$ which are homotopically identical.
If $g_{1}:X_{1}\rightarrow X_{2}$ and $g_{2}:X_{2}\rightarrow X_{3}$
are continuous operators such that $g_{2}\circ g_{1}$ is homotopic
to the identity on $X_{1}$, then $\cat X_{1}\leq\cat X_{2}$ . 
\end{rem}

\section{The map $\Phi_{\varepsilon}$ }

For every $\xi\in M$ we define the function
\begin{equation}
W_{\xi,\varepsilon}(x)=U_{\varepsilon}(\exp_{\xi}^{-1}x)\chi(|\exp_{\xi}^{-1}x|)\label{eq:Weps}
\end{equation}
where $\chi:\mathbb{R}^{+}\rightarrow\mathbb{R}^{+}$ is a cut off
function, that is $\chi\equiv1$ for $t\in[0,r/2)$, $\chi\equiv0$
for $t>r$ and $|\chi'(t)|\le2/r$. Here $\exp_{\xi}$ are the normal
coordinates centered in $\xi\in M$ and $r$ is the injectivity radius
of $M$. We recall the following well known expansion of the metric
$g$ in normal coordinates:
\begin{eqnarray}
g_{ij}(\varepsilon z)=\delta_{ij}+o(\varepsilon|z|) &  & |g(\varepsilon z)|^{\frac{1}{2}}=1+o(\varepsilon|z|)\label{eq:gexp}
\end{eqnarray}

We can define a map
\begin{eqnarray*}
\Phi_{\varepsilon} & : & M\rightarrow{\mathcal N}_{\varepsilon}\\
\Phi_{\varepsilon}(\xi) & = & t_{\varepsilon}(W_{\xi,\varepsilon})W_{\xi,\varepsilon}
\end{eqnarray*}

\begin{rem}
\label{w}We have that $W_{\varepsilon,\xi}\in H^{1}(M)$ and the
following limits hold uniformly with respect to $\xi\in M$
\begin{eqnarray*}
\|W_{\varepsilon,\xi}\|_{\varepsilon} & \rightarrow & \|U\|_{H^{1}(\mathbb{R}^{3})}\\
|W_{\varepsilon,\xi}|_{\varepsilon,t} & \rightarrow & \|U\|_{L^{t}(\mathbb{R}^{3})}\text{ for all }2\le t\le6
\end{eqnarray*}
\end{rem}
\begin{lem}
\label{lem:stimaGeps}We have that 
\[
\lim_{\varepsilon\rightarrow0}G_{\varepsilon}(W_{\varepsilon,\xi})=G(U)=\int_{\mathbb{R}^{3}}qU^{2}\psi(U)dx
\]
 uniformly with respect to $\xi\in M$\end{lem}
\begin{proof}
We set, for the sake of simplicity, $\psi_{\varepsilon}(x):=\psi_{\varepsilon}(W_{\varepsilon,\xi})(x)$,
and we define 
\[
\tilde{\psi}_{\varepsilon}(z)=\psi_{\varepsilon}(\exp_{\xi}(\varepsilon z))\chi_{r}(|\varepsilon z|)\text{ for }z\in\mathbb{R}^{3}
\]
It is easy to see that $\|\tilde{\psi}_{\varepsilon}\|_{H^{1}(\mathbb{R}^{3})}\le C\|\psi_{\varepsilon}\|_{\varepsilon}.$
Moreover, by (\ref{eq:psi-eps})
\[
\|\psi_{\varepsilon}\|_{\varepsilon}^{2}\le C\|W_{\varepsilon,\xi}\|_{\frac{12}{5},\varepsilon}^{2}\|\psi_{\varepsilon}\|_{\varepsilon}\le C\|U\|_{\frac{12}{5}}^{2}\|\psi_{\varepsilon}\|_{\varepsilon}
\]
so $\tilde{\psi}_{\varepsilon}$ is bounded in $H^{1}(\mathbb{R}^{3})$
so there exists $\bar{\psi}\in H^{1}(\mathbb{R}^{3})$ such that,
up to extracting a subsequence, $\tilde{\psi}_{\varepsilon_{k}}\rightharpoonup\bar{\psi}$
weakly in $H^{1}(\mathbb{R}^{3})$. 

First, we want to prove that $\bar{\psi}$ is a weak solution of
\[
-\Delta v+v=qU^{2}
\]
that is $\bar{\psi}=\psi_{\infty}(U)$. Given $f\in C_{0}^{\infty}(\mathbb{R}^{3})$,
we have that the $\text{spt }f\subset B(0,T)$ for some $T>0$, so
eventually $ $$\text{spt }f\subset B(0,r/\varepsilon_{k})$. Thus
we can define 
\[
f_{k}(x):=f\left(\frac{1}{\varepsilon_{k}}\exp_{\xi}^{-1}(x)\right)
\]
and we have that $f_{k}(x)$ is compactly supported in $B_{g}(\xi,r)$.
By definition of $\psi_{\varepsilon}(x)$ we have that 
\begin{equation}
\int_{M}\varepsilon_{k}^{2}\nabla_{g}\psi_{\varepsilon_{k}}\nabla_{g}f_{k}+\psi_{\varepsilon_{k}}f_{k}d\mu_{g}=q\int_{M}W_{\varepsilon_{k},\xi}^{2}f_{k}d\mu_{g}.\label{eq:soldeb1}
\end{equation}
By the change of variables $x=\exp_{\xi}(\varepsilon_{k}z)$ and by
(\ref{eq:gexp}) we get 
\begin{align*}
\frac{1}{\varepsilon_{k}^{3}}\int_{M}\varepsilon_{k}^{2}\nabla_{g}\psi_{\varepsilon_{k}}\nabla_{g}f_{k}+\psi_{\varepsilon_{k}}f_{k}d\mu_{g}= & \int_{B(0,r/\varepsilon_{k})}\left[g_{ij}(\varepsilon_{k}z)\partial_{i}\tilde{\psi}_{\varepsilon_{k}}(z)\partial_{j}f(z)+\tilde{\psi}_{\varepsilon_{k}}(z)f(z)\right]|g(\varepsilon_{k}z)|^{\frac{1}{2}}dz\\
= & \int_{B(0,T)}\nabla\tilde{\psi}_{\varepsilon_{k}}(z)\nabla f(z)+\tilde{\psi}_{\varepsilon_{k}}(z)f(z)dz+o(\varepsilon_{k})
\end{align*}
thus, by weak convergence of $\tilde{\psi}_{\varepsilon}$ we get
\begin{equation}
\frac{1}{\varepsilon_{k}^{3}}\int_{M}\varepsilon_{k}^{2}\nabla_{g}\psi_{\varepsilon_{k}}\nabla_{g}f_{k}+\psi_{\varepsilon_{k}}f_{k}d\mu_{g}\rightarrow\int_{\mathbb{R}^{3}}\nabla\bar{\psi}(z)\nabla f(z)+\bar{\psi}(z)f(z)dz.\label{eq:conv1}
\end{equation}
as $\varepsilon_{k}\rightarrow0$. In the same way we get 
\[
\frac{q}{\varepsilon_{k}^{3}}\int_{M}W_{\varepsilon_{k},\xi}^{2}f_{k}d\mu_{g}=q\int_{B(0,r/\varepsilon_{k})}U^{2}(z)f(z)|g(\varepsilon_{k}z)|^{\frac{1}{2}}dz=q\int_{\mathbb{R}^{3}}U^{2}(z)f(z)dz+o(\varepsilon_{k})
\]
and 
\begin{equation}
\frac{q}{\varepsilon_{k}^{3}}\int_{M}W_{\varepsilon_{k},\xi}^{2}f_{k}d\mu_{g}\rightarrow q\int_{\mathbb{R}^{3}}U^{2}(z)f(z)dz.\label{eq:conv2}
\end{equation}
 By (\ref{eq:soldeb1}), (\ref{eq:conv1}), (\ref{eq:conv2}) we get
that, for any $f\in C_{0}^{\infty}(\mathbb{R}^{3})$ it holds
\[
\int_{\mathbb{R}^{3}}\nabla\bar{\psi}\nabla f+\bar{\psi}f=q\int_{\mathbb{R}^{3}}U^{2}f
\]
 which proves that
\begin{equation}
\tilde{\psi}_{\varepsilon_{k}}\rightharpoonup\psi_{\infty}(U)\text{ weakly in }H^{1}(\mathbb{R}^{3})\label{eq:convdebpsibar}
\end{equation}
To conclude, again by change of variables we have
\[
G_{\varepsilon_{k}}(W_{\varepsilon_{k},\xi})=\frac{1}{\varepsilon_{k}^{3}}\int_{B_{g}(\xi,r)}W_{\varepsilon_{k},\xi}^{2}\psi(W_{\varepsilon_{k},\xi})d\mu_{g}=\int_{\mathbb{R}^{3}}U^{2}(z)\chi^{2}(|\varepsilon_{k}z|)\tilde{\psi}_{\varepsilon_{k}}|g(\varepsilon_{k}z)|^{\frac{1}{2}}dz.
\]
Since $U^{2}\in L^{6/5}(\mathbb{R}^{3})$ one has $ $
\[
U^{2}(z)\chi^{2}(|\varepsilon_{k}z|)|g(\varepsilon_{k}z)|^{\frac{1}{2}}\rightarrow U^{2}(z)\text{ strongly in }L^{6/5}(\mathbb{R}^{3}),
\]
that, combined with (\ref{eq:convdebpsibar}) concludes the proof.\end{proof}
\begin{prop}
\label{prop:phieps}For all $\varepsilon>0$ the map $\Phi_{\varepsilon}$
is continuous. Moreover for any $\delta>0$ there exists $\varepsilon_{0}=\varepsilon_{0}(\delta)$
such that, if $\varepsilon<\varepsilon_{0}$ then $I_{\varepsilon}\left(\Phi_{\varepsilon}(\xi)\right)<m_{\infty}+\delta$.\end{prop}
\begin{proof}
It is easy to see that $\Phi_{\varepsilon}$ is continuous because
$t_{\varepsilon}(w)$ depends continuously on $w\in H_{g}^{1}(M)$.

At this point we prove that $t_{\varepsilon}(W_{\varepsilon,\xi})\rightarrow1$
uniformly with respect to $\xi\in M$. In fact, by Lemma \ref{lem:nehari}
$t_{\varepsilon}(W_{\varepsilon,\xi})$ is the unique solution of
\[
t^{2}\|W_{\varepsilon,\xi}\|_{\varepsilon}^{2}+\omega G_{\varepsilon}(tW_{\varepsilon,\xi})-t^{p}|W_{\varepsilon,\xi}|_{\varepsilon,p}^{p}=0
\]
which, in light of Remark \ref{rem:tiquadro} can by rewritten as
\[
\|W_{\varepsilon,\xi}\|_{\varepsilon}^{2}+\omega t^{2}G_{\varepsilon}(W_{\varepsilon,\xi})-t^{p-2}|W_{\varepsilon,\xi}|_{\varepsilon,p}^{p}=0
\]
By Remark \ref{w} and Lemma \ref{lem:stimaGeps} we have the claim.
In fact, we recall that, since $U$ is a solution of (\ref{eq:PL})
it holds $\|U\|_{H^{1}(\mathbb{R}^{3})}^{2}+\omega G(U)-|U|_{L^{p}(\mathbb{R}^{3})}^{p}=0$.

At this point, we have 
\[
I_{\varepsilon}\left(t_{\varepsilon}(W_{\varepsilon,\xi})W_{\varepsilon,\xi}\right)=\left(\frac{1}{2}-\frac{1}{p}\right)\|W_{\varepsilon,\xi}\|_{\varepsilon}^{2}t_{\varepsilon}^{2}+\omega\left(\frac{1}{4}-\frac{1}{p}\right)t_{\varepsilon}^{4}G_{\varepsilon}(W_{\varepsilon,\xi})
\]
Again, by Remark \ref{w} and Lemma \ref{lem:stimaGeps} and since
$t_{\varepsilon}(W_{\varepsilon,\xi})\rightarrow1$ we have

\[
I_{\varepsilon}\left(t_{\varepsilon}(W_{\varepsilon,\xi})W_{\varepsilon,\xi}\right)\rightarrow\left(\frac{1}{2}-\frac{1}{p}\right)\|U\|_{H^{1}(\mathbb{R}^{3})}^{2}+\omega\left(\frac{1}{4}-\frac{1}{p}\right)G(U)=m_{\infty}
\]
that concludes the proof.\end{proof}
\begin{rem}
\label{rem:limsup}We set
\[
m_{\varepsilon}=\inf_{{\mathcal N}_{\varepsilon}}I_{\varepsilon.}
\]
By Proposition \ref{prop:phieps} we have that 

\begin{equation}
\limsup_{\varepsilon\rightarrow0}m_{\varepsilon}\le m_{\infty.}\label{eq:limsup}
\end{equation}

\end{rem}

\section{The map $\beta$}

For any $u\in{\mathcal N}_{\varepsilon}$ we can define a point $\beta(u)\in\mathbb{R}^{n}$
by 
\[
\beta(u)=\frac{\int_{M}x\Gamma(u)d\mu_{g}}{\int_{M}\Gamma(u)d\mu_{g}}
\]
where $\Gamma(u)=\left(\frac{1}{2}-\frac{1}{p}\right)\frac{1}{\varepsilon^{3}}|u^{+}|^{p}-\frac{\omega}{4}\frac{1}{\varepsilon^{3}}u^{2}\psi_{\varepsilon}(u)$.
Immediately one has that the function $\beta$ is well defined in
${\mathcal N}_{\varepsilon}$, since $\int_{M}\Gamma(u)d\mu_{g}=I_{\varepsilon}(u)\ge m_{\varepsilon}$
\begin{lem}
There exists $\alpha>0$ such that $m_{\varepsilon}\ge\alpha$ for
all $\varepsilon$. \end{lem}
\begin{proof}
Take $w$ such that $|w^{+}|_{\varepsilon,p}=1$, and $t_{\varepsilon}=t_{\varepsilon}(w)$
such that $t_{\varepsilon}w\in{\mathcal N}_{\varepsilon}$. By (\ref{eq:I-nehari})
we have
\[
I_{\varepsilon}(t_{\varepsilon}w)=\frac{t_{\varepsilon}^{2}}{4}\|w\|_{\varepsilon}^{2}+\left(\frac{1}{4}-\frac{1}{p}\right)t_{\varepsilon}^{p}\ge\left(\frac{1}{4}-\frac{1}{p}\right)t_{\varepsilon}^{p}.
\]
Moreover, we have that $\inf_{|w^{+}|_{\varepsilon,p}=1}t_{\varepsilon}(w)>0$.
In fact, suppose that there exists a sequence $w_{n}$ such that $|w^{+}|_{\varepsilon,p}=1$
and $t_{\varepsilon}(w_{n})\rightarrow0$. Since $t_{\varepsilon}(w_{n})w_{n}\in{\mathcal N}_{\varepsilon}$
it holds
\[
1=|w_{n}^{+}|_{\varepsilon,p}=\frac{1}{t_{\varepsilon}(w_{n})^{p-2}}\|w_{n}\|_{\varepsilon}^{2}+\omega G_{\varepsilon}(t_{\varepsilon}(w_{n}))\ge\frac{1}{t_{\varepsilon}(w_{n})^{p-2}}\|w_{n}\|_{\varepsilon}^{2}.
\]
Also, we have that there exists a constant $C>0$ which does not depend
on $\varepsilon$ such that $|w_{n}^{+}|_{\varepsilon,p}\le|w_{n}|_{\varepsilon,p}\le C\|w_{n}\|_{\varepsilon}$,
so 
\[
1\ge\frac{1}{Ct_{\varepsilon}(w_{n})^{p-2}}\rightarrow+\infty
\]
that is a contradiction. This proves that $m_{\varepsilon}\ge\alpha$
for some $\alpha>0$.
\end{proof}
Now we have to prove that, if $u\in{\mathcal N}_{\varepsilon}\cap I_{\varepsilon}^{m_{\infty}+\delta}$
then $\beta(u)\in M^{+}$.

Let us consider the following partitions of $M$. For a given $\varepsilon>0$
we say that a finite partition${\mathcal P}_{\varepsilon}=\left\{ P_{j}^{\varepsilon}\right\} _{j\in\Lambda_{\varepsilon}}$
of $M$ is a \textquotedblleft good\textquotedblright{} partition
if: for any $j\in\Lambda_{\varepsilon}$ the set $P_{j}^{\varepsilon}$
is closed; $P_{i}^{\varepsilon}\cap P_{j}^{\varepsilon}\subset\partial P_{i}^{\varepsilon}\cap\partial P_{j}^{\varepsilon}$
for any $i\ne j$; there exist $r_{1}(\varepsilon),r_{2}(\varepsilon)>0$
such that there are points $q_{j}^{\varepsilon}\in P_{j}^{\varepsilon}$
for which  $B_{g}(q_{j}^{\varepsilon},\varepsilon)\subset P_{j}^{\varepsilon}\subset B_{g}(q_{j}^{\varepsilon},r_{2}(\varepsilon))\subset B_{g}(q_{j}^{\varepsilon},r_{1}(\varepsilon))$,
with $r_{1}(\varepsilon)\ge r_{2}(\varepsilon)\ge C\varepsilon$ for
some positive constant $C$; lastly, there exists a finite number
$\nu(M)\in\mathbb{N}$ such that every $\xi\in M$ is contained in
at most $\nu(M)$ balls $B_{g}(q_{j}^{\varepsilon},r_{1}(\varepsilon))$,
where $\nu(M)$ does not depends on $\varepsilon$.
\begin{rem}
\label{lem:gamma} We recall that there exists a constant $\gamma>0$
such that, for any $\delta>0$ and for any $\varepsilon<\varepsilon_{0}(\delta)$
as in Proposition \ref{prop:phieps}, given any ``good'' partition
${\mathcal P}_{\varepsilon}=\left\{ P_{j}^{\varepsilon}\right\} _{j}$
of the manifold $M$ and for any function $u\in{\mathcal N}_{\varepsilon}\cap I_{\varepsilon}^{m_{\infty}+\delta}$
there exists, for an index $\bar{j}$ a set $P_{\bar{j}}^{\varepsilon}$
such that 
\begin{equation}
\frac{1}{\varepsilon^{3}}\int_{P_{\bar{j}}^{\varepsilon}}|u^{+}|^{p}dx\ge\gamma.\label{eq:gamma}
\end{equation}
Indeed we can proceed verbatim as in Lemma 12 of \cite{GMruf}, considering
that, since $I'(u)[u]=0$,
\begin{eqnarray*}
\|u\|_{\varepsilon}^{2} & = & |u^{+}|_{\varepsilon,p}^{p}-\frac{1}{\varepsilon^{3}}\int_{M}\omega u^{2}\psi(u)\le|u^{+}|_{\varepsilon,p}^{p}\\
 & = & \sum_{j}|u_{j}^{+}|_{\varepsilon,p}^{p}\le\max_{j}\left\{ |u_{j}^{+}|_{\varepsilon,p}^{p-2}\right\} \sum_{j}|u_{j}^{+}|_{\varepsilon,p}^{2}
\end{eqnarray*}
where $u_{j}^{+}$ is the restriction of the function $u^{+}$on the
set $P_{j}$, and arguing as in Lemma 5.3 of \cite{BBM}, we obtain
that for some $C>0$ it holds $\sum_{j}|u_{j}^{+}|_{\varepsilon,p}^{2}\le C\nu\|u^{+}\|_{\varepsilon}^{2},$
and there the proof follows since 
\[
\max_{j}\left\{ |u_{j}^{+}|_{\varepsilon,p}^{p-2}\right\} \ge\frac{1}{C\nu}.
\]
\end{rem}
\begin{prop}
\label{prop:conc}For any $\eta\in(0,1)$ there exists $\delta_{0}<m_{\infty}$
such that for any $\delta\in(0,\delta_{0})$ and any $\varepsilon\in(0,\varepsilon_{0}(\delta))$
as in Proposition \ref{prop:phieps}, for any function $u\in{\mathcal N}_{\varepsilon}\cap I_{\varepsilon}^{m_{\infty}+\delta}$
we can find a point $q=q(u)\in M$ such that 
\[
\int_{B_{g}(q,r/2)}\Gamma(u)>\left(1-\eta\right)m_{\infty}.
\]
\end{prop}
\begin{proof}
First, we prove the proposition for $u\in{\mathcal N}_{\varepsilon}\cap I_{\varepsilon}^{m_{\varepsilon}+2\delta}$. 

By contradiction, we assume that there exists $\eta\in(0,1)$ such
that we can find two sequences of vanishing real number $\delta_{k}$
and $\varepsilon_{k}$ and a sequence of functions $\left\{ u_{k}\right\} _{k}$
such that $u_{k}\in{\mathcal N}_{\varepsilon_{k}}$, 
\begin{equation}
m_{\varepsilon_{k}}\le I_{\varepsilon_{k}}(u_{k})=\left(\frac{1}{2}-\frac{1}{p}\right)\|u_{k}\|_{\varepsilon_{k}}^{2}+\omega\left(\frac{1}{4}-\frac{1}{p}\right)G_{\varepsilon_{k}}(u_{k})\le m_{\varepsilon_{k}}+2\delta_{k}\le m_{\infty}+3\delta_{k}\label{eq:mepsk}
\end{equation}
 for $k$ large enough (see Remark \ref{rem:limsup}), and, for any
$q\in M$, 
\[
\int_{B_{g}(q,r/2)}\Gamma(u_{k})\le\left(1-\eta\right)m_{\infty}.
\]
By Ekeland principle and by definition of ${\mathcal N}_{\varepsilon_{k}}$
we can assume 
\begin{equation}
\left|I'_{\varepsilon_{k}}(u_{k})[\varphi]\right|\le\sigma_{k}\|\varphi\|_{\varepsilon_{k}}\text{ where }\sigma_{k}\rightarrow0.\label{eq:ps}
\end{equation}
By Remark \ref{lem:gamma} there exists a set $P_{k}^{\varepsilon_{k}}\in{\mathcal P}_{\varepsilon_{k}}$
such that 
\[
\frac{1}{\varepsilon_{k}^{3}}\int_{P_{k}^{\varepsilon_{k}}}|u_{k}^{+}|^{p}d\mu_{g}\ge\gamma,
\]
 so, we choose a point $q_{k}\in P_{k}^{\varepsilon_{k}}$ and we
define, in analogy with the proof of Lemma \ref{lem:stimaGeps}

\[
w_{k}(z):=u_{k}(\exp_{q_{k}}(\varepsilon_{k}z))\chi(\varepsilon_{k}|z|)
\]
where $z\in B(0,r/\varepsilon_{k})\subset\mathbb{R}^{3}$. Extending
trivially $w_{k}$ by zero to the whole $\mathbb{R}^{3}$ we have
that $w_{k}\in H^{1}(\mathbb{R}^{3})$ and, by (\ref{eq:mepsk}),
\[
\|w_{k}\|_{H^{1}(\mathbb{R}^{3})}^{2}\le C\|u_{k}\|_{\varepsilon_{k}}^{2}\le C.
\]
So there exists a $w\in H^{1}(\mathbb{R}^{3})$ such that, up to subsequences,
$w_{k}\rightarrow w$ weakly in $H^{1}(\mathbb{R}^{3})$ and strongly
in $L_{\text{loc}}^{t}(\mathbb{R}^{3})$ for $2\le t<6$. Moreover
we set $\psi_{k}(x):=\psi_{\varepsilon}(u_{k})(x)$ and $\tilde{\psi}_{k}=\psi_{k}(\exp_{q_{k}}(\varepsilon_{k}z))\chi(\varepsilon_{k}|z|)$.
Arguing as in Lemma \ref{lem:stimaGeps} we get that $\tilde{\psi}_{k}\rightarrow\psi_{\infty}(w)$
weakly in $H^{1}(\mathbb{R}^{3})$ and strongly in $L_{\text{loc}}^{t}(\mathbb{R}^{3})$
for all $2\le t<6$. 

Again, given $f\in C_{0}^{\infty}(\mathbb{R}^{3})$, with $\text{spt }f\subset B(0,T)$
for some $T>0$ we can define 
\[
f_{k}(x):=f\left(\frac{1}{\varepsilon_{k}}\exp_{\xi}^{-1}(x)\right)
\]
and, by (\ref{eq:ps}) we have $\left|I'_{\varepsilon_{k}}(u_{k})[f_{k}]\right|\rightarrow0$
as $k\rightarrow\infty$. Now, by change of variables we have 
\begin{align*}
I'_{\varepsilon_{k}}(u_{k})[f_{k}]= & \frac{1}{\varepsilon_{k}^{3}}\int_{M}\varepsilon_{k}^{2}\nabla_{g}u_{k}\nabla_{g}f_{k}+u_{k}f_{k}+\omega qu_{k}\psi_{k}f_{k}-(u_{k}^{+})^{p-1}f_{k}d\mu_{g}\\
= & \int_{B(0,T)}\left[g_{ij}(\varepsilon_{k})\partial_{i}w_{k}\partial_{j}f+w_{k}f+\omega qw_{k}\tilde{\psi}_{k}f-(w_{k}^{+})^{p-1}f\right]|g(\varepsilon_{k}z)|^{\frac{1}{2}}dz\\
= & \int_{\mathbb{R}^{3}}\nabla w_{k}\nabla f+w_{k}f+\omega qw_{k}\tilde{\psi}_{k}f-(w_{k}^{+})^{p-1}fdz+o(\varepsilon_{k})\\
\rightarrow & \int_{\mathbb{R}^{3}}\nabla w\nabla f+wf+\omega qw\psi_{\infty}(w)f-(w_{k}^{+})^{p-1}fdz=I'_{\infty}(w)[f]
\end{align*}
and, by (\ref{eq:ps}), we get that $w$ is a weak solution of the
limit problem (\ref{eq:PL}) and that $w\in\mathcal{N}_{\infty}$.
By Lemma \ref{lem:gamma} and by the choice of $q_{k}$ we have that
$w\ne0$, so $w>0$ and $I_{\infty}(w)\ge m_{\infty}$. 

Now, consider the functions 
\[
h_{k}:=\frac{1}{\varepsilon^{3}}|u_{k}^{+}|^{\frac{1}{p}}(\exp_{q_{k}}(\varepsilon_{k}z))|g_{q_{k}}(\varepsilon_{k}z)|^{\frac{1}{2p}}\mathbb{I}_{B_{g}(q_{k},r)}
\]
 where $\mathbb{I}_{B_{g}(q_{k},r)}$ is the indicatrix function on
$B_{g}(q_{k},r)$. Since $|u_{k}|_{\varepsilon,p}$ is bounded, then
$h_{k}$ is bounded in $L^{p}(\mathbb{R}^{3})$ so, it converges weakly
to some $\bar{h}\in L^{p}(\mathbb{R}^{3})$. We have that $h=|w^{+}|^{\frac{1}{p}}$.
Take $f\in C_{0}^{\infty}(\mathbb{R}^{3})$, with $\text{spt }f\subset B(0,T)$
for some $T>0$. Since, eventually $B(0,T)\subset B(0,r/2\varepsilon_{k})$,
$|u_{k}^{+}|^{\frac{1}{p}}(\exp_{q_{k}}(\varepsilon_{k}z))=w_{k}^{+}$
on $B(0,T)$. Moreover, on $B(0,T)$ we have that $|g_{q_{k}}(\varepsilon_{k}z)|^{\frac{1}{2p}}=1+o(\varepsilon_{k})$.
Thus, since $w_{k}\rightharpoonup w$ in $L^{p}(\mathbb{R}^{3})$
we get. 
\[
\int_{\mathbb{R}^{3}}h_{k}fdz\rightarrow\int_{\mathbb{R}^{3}}|w^{+}|^{\frac{1}{p}}fdz
\]
for any $f\in C_{0}^{\infty}(\mathbb{R}^{3})$. In the same way we
can consider the functions 
\[
j_{k}=\frac{1}{\varepsilon^{3}}\left(g_{ij}(\varepsilon_{k}z)\partial_{i}u_{k}(\exp_{q_{k}}(\varepsilon_{k}z))\partial_{j}u_{k}(\exp_{q_{k}}(\varepsilon_{k}z))|g_{q_{k}}(\varepsilon_{k}z)|^{\frac{1}{2}}\right)^{\frac{1}{2}}\mathbb{I}_{B_{g}(q_{k},r)}
\]
\[
l_{k}:=\frac{1}{\varepsilon^{3}}|u_{k}|^{\frac{1}{2}}(\exp_{q_{k}}(\varepsilon_{k}z))|g_{q_{k}}(\varepsilon_{k}z)|^{\frac{1}{4}}\mathbb{I}_{B_{g}(q_{k},r)}
\]
We have that $j_{k},l_{k}\in L^{2}(\mathbb{R}^{3})$ and that $j_{k}\rightharpoonup|\nabla w|^{\frac{1}{2}}$,
$l_{k}\rightharpoonup|w|^{\frac{1}{2}}$ in $L^{2}(\mathbb{R}^{3})$.
Thus we have 

At this point, since $w\in\mathcal{N}_{\infty}$ and by (\ref{eq:mepsk})
we get 
\begin{align*}
m_{\infty}\le & I_{\infty}(w)=\frac{1}{4}\|w\|_{H^{1}}^{2}+\left(\frac{1}{4}-\frac{1}{p}\right)|w^{+}|_{p}^{p}\\
\le & \liminf_{k\rightarrow\infty}\frac{1}{4}\|j_{k}\|_{L^{2}}^{2}+\frac{1}{4}\|i_{k}\|_{L^{2}}^{2}+\left(\frac{1}{4}-\frac{1}{p}\right)|h_{k}|_{p}^{p}\\
\le & \frac{1}{4}\|u_{k}\|_{\varepsilon}^{2}+\left(\frac{1}{4}-\frac{1}{p}\right)|u_{k}^{+}|_{p}^{p}\le m_{\infty}+3\delta_{k}
\end{align*}
so we have that $w$ is a ground state for the limit problem (\ref{eq:PL}).

Given $T>0$, by the definition of $w_{k}$ we get, for $k$ large
enough

\begin{multline}
\int_{B(0,T)}\left[\left(\frac{1}{2}-\frac{1}{p}\right)\left(w_{k}^{+}\right)^{p}-\frac{\omega}{4}w_{k}^{2}\tilde{\psi}_{k}\right]g(\varepsilon_{k}z)dz\\
=\frac{1}{\varepsilon^{3}}\int_{B(q_{k},\varepsilon_{k}T)}\left(\frac{1}{2}-\frac{1}{p}\right)\left(u_{k}^{+}\right)^{p}-\frac{\omega}{4}u_{k}^{2}\psi_{\varepsilon}(u_{k})d\mu_{g}\\
=\int_{B(q_{k},\varepsilon_{k}T)}\Gamma(u_{k})dx\le\int_{B(q_{k},r/2)}\Gamma(u_{k})dx\le\left(1-\eta\right)m_{\infty}\label{eq:contr}
\end{multline}
and, if we choose $T$ sufficiently big, this leads to a contradiction
since $w_{k}\rightarrow w$ and $\tilde{\psi}_{k}\rightarrow\psi_{\infty}(w)$
in $L^{t}(B(0,T))$ for any $T>0$. Since $m_{\infty}=I_{\infty}(w)=\left(\frac{1}{2}-\frac{1}{p}\right)|w^{+}|^{p}-\frac{\omega}{4}G(w)$,
it is possible to choose $T$ such that (\ref{eq:contr})$ $ is false,
so the lemma is proved for $u\in{\mathcal N}_{\varepsilon}\cap I_{\varepsilon}^{m_{\varepsilon}+2\delta}$.

The above arguments also prove that 
\[
\liminf_{k\rightarrow\infty}m_{\varepsilon_{k}}\ge\lim_{k\rightarrow\infty}I_{\varepsilon_{k}}(u_{k})=m_{\infty}.
\]
and, in light of (\ref{eq:limsup}), this leads to 
\begin{equation}
\lim_{\varepsilon\rightarrow0}m_{\varepsilon}=m_{\infty}.\label{eq:mepsminfty}
\end{equation}
Hence, when $\varepsilon,\delta$ are small enough, ${\mathcal N}_{\varepsilon}\cap I_{\varepsilon}^{m_{\infty}+\delta}\subset{\mathcal N}_{\varepsilon}\cap I_{\varepsilon}^{m_{\varepsilon}+2\delta}$
and the general claim follows.\end{proof}
\begin{prop}
There exists $\delta_{0}\in(0,m_{\infty})$ such that for any $\delta\in(0,\delta_{0})$
and any $\varepsilon\in(0,\varepsilon(\delta_{0})$ (see Proposition
\ref{prop:phieps}), for every function $u\in{\mathcal N}_{\varepsilon}\cap I_{\varepsilon}^{m_{\infty}+\delta}$
it holds $\beta(u)\in M^{+}$. Moreover the composition 
\[
\beta\circ\Phi_{\varepsilon}:M\rightarrow M^{+}
\]
 is s homotopic to the immersion $i:M\rightarrow M^{+}$ \end{prop}
\begin{proof}
By Proposition \ref{prop:conc}, for any function $u\in{\mathcal N}_{\varepsilon}\cap I_{\varepsilon}^{m_{\infty}+\delta}$,
for any $\eta\in(0,1)$ and for $\varepsilon,\delta$ small enough,
we can find a point $q=q(u)\in M$ such that 
\[
\int_{B(q,r/2)}\Gamma(u)>\left(1-\eta\right)m_{\infty}.
\]
Moreover, since $u\in{\mathcal N}_{\varepsilon}\cap I_{\varepsilon}^{m_{\infty}+\delta}$
we have 
\[
I_{\varepsilon}(u)=\int_{M}\Gamma(u)\le m_{\infty}+\delta.
\]
Hence 
\begin{eqnarray*}
|\beta(u)-q| & \le & \frac{\left|\int_{M}(x-q)\Gamma(u)\right|}{\int_{M}\Gamma(u)}\\
 & \le & \frac{\left|\frac{1}{\varepsilon^{3}}\int_{B(q,r/2)}(x-q)\Gamma(u)\right|}{\int_{M}\Gamma(u)}+\frac{\left|\frac{1}{\varepsilon^{3}}\int_{M\smallsetminus B(q,r/2)}(x-q)\Gamma(u)\right|}{\int_{M}\Gamma(u)}\\
 & \le & \frac{r}{2}+2\text{diam}(M)\left(1-\frac{1-\eta}{1+\delta/m_{\infty}}\right),
\end{eqnarray*}
and the second term can be made arbitrarily small, choosing $\eta$,
$\delta$ and $\varepsilon$ sufficiently small. The second claim
of the theorem is standard.
\end{proof}

\paragraph{Acknowledgments}

Every problem has a story and an inspiration. This one comes out from
an interesting remark of Professor Ireneo Peral during a conference
in Alghero (Italy). We hence would like to thank Prof. Peral for his
keen and inspiring suggestion.

\end{document}